\documentclass[graybox]{svmult}

\usepackage{amsfonts,amstext,amssymb,physics} 
\usepackage{amsmath}

\usepackage{amssymb}
\usepackage{mathptmx}       
\usepackage{helvet}         
\usepackage{courier}        
\usepackage{type1cm}        
\usepackage{graphicx}        
\usepackage{multicol}        
\usepackage[bottom]{footmisc}

\usepackage{enumerate}

\usepackage{enumitem}  
\usepackage{calc}
\setlist{labelindent=1pt,itemsep=0.1cm}
\setlist[itemize]{leftmargin=0.7cm}
\setlist[enumerate]{itemindent=0em,leftmargin=0.7cm}

\usepackage{float} 

\usepackage[bookmarks=false,hidelinks,unicode]{hyperref}

\smartqed

\allowdisplaybreaks

\begin{document}
\title*{Fixed point results for Pre\v{s}i\'c type contractive mappings in $b$-metric spaces}
\titlerunning{Fixed point results for Pre\v{s}i\'c type contractive mappings in $b$-metric spaces} 
\author{Talat Nazir \and Sergei Silvestrov}
\authorrunning{T. Nazir, S. Silvestrov} 

\institute{Talat Nazir,
\at
Department of Mathematical Sciences, University of South Africa, Florida 0003, South Africa.\\ \email{talatn@unisa.ac.za}
\and
Sergei Silvestrov
\at Division of Mathematics and Physics, School of Education, Culture and Communication, M{\"a}lardalen University, Box 883, 72123 V{\"a}ster{\aa}s, Sweden. \\ \email{sergei.silvestrov@mdu.se}
\and }
%
%


\maketitle
\label{chap:NazirSilvestrov:FPRPMbMS}

\abstract*{The fixed point results for generalized Pre\v{s}i\'c type mappings in the setup of $b$-metric spaces are obtained. The stability of fixed point set of Pre\v{s}i\'c type mappings is also established. Several examples are also presented to illustrate the validity of the main results.
}

\abstract{The fixed point results for generalized Pre\v{s}i\'c type mappings in the setup of $b$-metric spaces are obtained. The stability of fixed point set of Pre\v{s}i\'c type mappings is also established. Several examples are also presented to illustrate the validity of the main results.
}

\section{Introduction to Pre\v{s}i\'c type Mappings and Fixed Points}
\label{secNaSe:IntroductionPresic}
The study of fixed point of maps with certain type of
contractive restrictions is a powerful approach towards solving variety of scientific problems in various areas of mathematics as well as important methodology for computational algorithms in the natural sciences and engineering subjects.

Banach contraction principle \cite{Banach} is a simple powerful
result with a wide range of applications, including iterative methods for solving linear, nonlinear, differential, integral, and difference equations.
Banach \cite{Banach} initiated the study of fixed point theory for contraction mappings. There are several generalizations and extensions of the Banach contraction principle in the existing literature. Let us firstly recall the Banach contraction principle.
Henceforth, $\mathbb{R}$, $\mathbb{R}_{\geq 0}$, $\mathbb{R}_{>0}$, $\mathbb{Z}_{\geq0}$ and $\mathbb{Z}_{> 0}$ denote the sets of real numbers, non-negative real numbers, positive real numbers, non-negative integers and positive integers, respectively.
\begin{theorem}[Banach contraction principle, \cite{Banach}]
	\label{NaSeTheorem1.1.1.} If $(X,d)$ is a complete
	metric space and mapping $f:X\rightarrow X$ satisfies
	\begin{equation*}
		d(fx,fy)\leq \eta \,d(x,y),\  \text{for all }x,y\in X,
	\end{equation*}
	where $0\leq \eta <1$, then there exists a unique
$u\in X$ that satisfies $u=fu$. Moreover, for any $x_{0}\in X$, the iterative
	sequence $x_{n+1}=f\left( x_{n}\right) $ converges to $u$.\end{theorem}
  Let $f:X^{k}\rightarrow X$, where $k\in \mathbb{Z}_{> 0}$.
A point $x^{\ast }\in X$ is called a fixed point of $f$ if $f(x^{\ast},\ldots ,x^{\ast })=x^{\ast }$. Consider the possibly non-linear $k$-th order
difference equation
\begin{equation}
	x_{n+k}=f(x_{n},\ldots ,x_{n+k-1}),\, \,n=1,2,\ldots  \label{t.1.1}
\end{equation}
with the initial values $x_{1},\ldots ,x_{k}\in X$.
The difference equation \eqref{t.1.1} can be studied by means of fixed point theory in
view of the fact that $x^{\ast }$ in $X$ is a solution of \eqref{t.1.1} if and only
if $x^{\ast }$ is a fixed point of the self-mapping $F:X\rightarrow X$ given
by
\begin{equation*}
	F(x)=f(x,\ldots ,x),\ \text{for all}\ x\in X.
\end{equation*}

One of the most important results in this direction is obtained by
Pre\v{s}i\'c \cite{Presic}.
\begin{theorem}
	\label{NaSeTheorem1.1.2.} Let $(X,d)$ be a complete metric space,
	$k\in \mathbb{Z}_{> 0}$, and let $f:X^{k}\rightarrow X$ be a mapping satisfying the
	following contractive type condition
	\begin{equation*}
		d(f(x_{1},\ldots ,x_{k}),f(x_{2},\ldots ,x_{k+1}))\leq \sum\limits_{j=1}^{k} r_{j}d(x_{j},x_{j+1}),
	\end{equation*}
	for every $x_{1},\ldots ,x_{k+1}\in X$, where $r_{1},\ldots ,r_{k}\in \mathbb{R}_{\geq 0}$
	are non-negative constants such that $r_{1}+\ldots +r_{k}<1$. Then
	there exists a unique point $u\in X$ such that $f(u,\ldots ,u)=u$. Moreover,
	for any arbitrary points $x_{1},\ldots ,x_{k}\in X$, the sequence \eqref{t.1.1}
	converges to $u$.\end{theorem}

  It is easy to show that for $k=1$, Theorem \ref{NaSeTheorem1.1.2.} reduces to the
Banach contraction principle.

  \'{C}iri\'{c} and Pre\v{s}i\'c \cite{CiricPresic07} generalized the
above theorem as follows.

\begin{theorem}
	\label{NaSeTheorem1.1.3.} Let $(X,d)$ be a complete metric space, $k\in \mathbb{Z}_{> 0}$, and let $f:X^{k}\rightarrow X$ be a mapping satisfying the
	following contractive type condition
	\begin{equation*}
		d(f(x_{1},\ldots ,x_{k}),f(x_{2},\ldots ,x_{k+1}))
        \leq \kappa \max\{d(x_{1},x_{2}),\ldots, d(x_{j},x_{j+1}),\ldots ,d(x_{k},x_{k+1})\},
	\end{equation*}
	for every $x_{1},\ldots ,x_{k+1}\in X$, where $0<\kappa<1$ is a constant. Then
	there exists $u\in X$ such that $f(u,\ldots ,u)=u$. Moreover, for any
	arbitrary points $x_{1},\ldots ,x_{k}\in X$, the sequence \eqref{t.1.1} is
	convergent and
	\begin{equation*}
		\lim \limits_{n\rightarrow \infty }x_{n}=f(\lim \limits_{n\rightarrow \infty
		}x_{n},\ldots ,\lim \limits_{n\rightarrow \infty }x_{n}).
	\end{equation*}
	If in addition we suppose that
	\begin{equation*}
		d(f(x,\ldots ,x),f(y,\ldots ,y))<d(x,y)
	\end{equation*}
	for all $x,y\in X$ with $x\neq y$, then $u$ is the unique point in $X$
	with $u=f(u,u,\dots,u)$.
\end{theorem}

  Abbas at al. \cite{AIN15} extended Theorem \ref{NaSeTheorem1.1.2.} and Theorem \ref{NaSeTheorem1.1.3.}, and
proved the following results for Pre\v{s}i\'c type weakly contractive mappings.

\begin{theorem}
	\label{NaSeTheorem1.1.4.} Let $(X,d)$ be a complete metric
	space, $k\in\mathbb{Z}_{>0}$ and $f:X^{k}\rightarrow X$ be a mapping. If there
	exists a lower semi-continuous function $\phi :\mathbb{R}_{\geq 0}\rightarrow \mathbb{R}_{\geq 0}$ with $\phi (t)=0$ only for $t=0$, such that
	\begin{multline*}
			d(f(x_{1},x_{2},\ldots ,x_{k}),f(x_{2},x_{3},\ldots ,x_{k+1}))\\
            \leq  \max\{d(x_{i},x_{i+1}):1\leq i\leq k\} -\phi (\max \{d(x_{i},x_{i+1}):1\leq i\leq k\}),
	\end{multline*}
	for all $(x_{1},\ldots ,x_{k+1})\in X^{k+1}$, then, for any points
	$x_{1},\ldots ,x_{k}\in X$, the sequence $\{x_{n}\}_{n\geq 1}$ defined by \eqref{t.1.1}
	converges to $u\in X$ when $n\to \infty$, and $u$ is a fixed point of $f$, that is, $u=f(u,\dots,u)$. If, moreover,
	\begin{equation*}
		d(f(x,\dots,x),f(y,\dots,y))=d(x,y)-\phi (d(x,y)),
	\end{equation*}
for all $x,y\in X$ with $x\neq y$, then $u$ is the unique fixed point
	of $f$.
\end{theorem}

\section{$b$-Metric Spaces}
Czerwik introduced $b$-metric spaces as generalization of metric spaces,
first for special case $b=2$ in \cite{C93} in 1993, and then for arbitrary real $b\geq1$ in \cite{C98} in 1998.
In $b$-metric space, the triangular inequality involves a constant
$b\geq 1$ called $b$-metric constant.
\begin{definition}
	\label{NaSeDefbmetricspace} Let $X\neq \emptyset $ be a set, and $b\in \mathbb{R}$, $b\geq 1$. A mapping $d:X\times X\rightarrow
	\mathbb{R}_{\geq 0}$ is called a {\it $b$-metric} on $X$ if for any $x,y,z\in X,$ the following
	conditions hold:
	\begin{enumerate}
		\item[(b$_{1}$)] \label{b1bmetrspdef} \quad $d(x,y)=0$ if and only if $x=y,$
		\item[(b$_{2}$)] \label{b2bmetrspdef} \quad $d(x,y)=d(y,x),$
		\item[(b$_{3}$)] \label{b3bmetrspdef} \quad $d(x,y)\leq b\left( d(x,z)+d(z,y)\right).$
	\end{enumerate}
	  The pair $(X,d)$ is called a {\it $b$-metric space} with parameter $b\in \mathbb{R}_{\geq 1}$.
\end{definition}

Note that, if $b=1,$ the Definition \ref{NaSeDefbmetricspace} coincides with that of a metric space.

\begin{example} \textrm{(\cite{RPA14})}
	\label{NaSeExample1.3.2.}  Let $(X,d)$ be a metric
	space, and $\rho (x,y)=(d(x,y))^{p},$ where $p\in \mathbb{R}$, $p>1$. Then $\rho $ is a $b$-metric with $b=2^{p-1}.$
	Obviously, the conditions \ref{b1bmetrspdef} and \ref{b2bmetrspdef} in the Definition \ref{NaSeDefbmetricspace}
	are satisfied. If $p \in \mathbb{R}, p>1$ then the convexity of the function $f(x)=x^{p}$ on $\mathbb{R}_{>0}$ implies
	\begin{equation*}
		\left( \frac{a+b}{2}\right) ^{p}\leq \frac{1}{2}\left( a^{p}+b^{p}\right),
	\end{equation*}
	and hence, $\left( a+b\right) ^{p}\leq 2^{p-1}(a^{p}+b^{p})$. This inequality and $f(x)=x^{p}$ being monotonically increasing on $\mathbb{R}_{>0}$
yield, for all $x,y,z\in X$,
\begin{eqnarray*}
		\rho (x,y) &=&(d(x,y))^{p}\leq (d(x,z)+d(z,y))^{p} \\
		&\leq &2^{p-1}(d(x,z)^{p}+d(z,y)^{p}) \\
		&=&2^{p-1}(\rho (x,z)+\rho (z,y)).
	\end{eqnarray*}
	So, the condition \ref{b3bmetrspdef} also holds, and hence $\rho $
	is a $b$-metric on $X$ with $b=2^{p-1}$.\end{example}

\begin{example}  If $X=\mathbb{R}$ and $d(x,y)=\vert x-y\vert $ is the
usual Euclidean metric, then $\rho (x,y)=(x-y)^{2}$ is a $b$-metric on $
\mathbb{R}
$ with $b=2,$ but is not a metric on $\mathbb{R}$.
\end{example}

\begin{example} \textbf{(\cite{Bakhtin09})} The space
\[
l_{p}=\{ \{x_{n}\}_{n\geq 1} \subset \mathbb{R} \mid \sum_{n=1}^{\infty }\left \vert x_{n}\right \vert ^{p}<\infty \},\ \text{for } 0<p<1
\]
is a $b$-metric space with $b=2^{\frac{1}{p}}$ and the $b$-metric $d:l_{p}\times l_{p}\rightarrow \mathbb{R}_{\geq 0}$ defined by
\[
d(x,y)=\left[ \sum_{n=1}^{\infty }\left \vert x_{n}-y_{n}\right \vert ^{p}\right] ^{\frac{1}{p}}.
\]
For $p\geq 1$, it is a metric space, and hence, it is a $b$-metric space for all $p > 0$.
\end{example}

\begin{example} \textbf{(\cite[\S 15.9, pages 157-158]{Koethe69TopvectspI}, \cite[Lemma 3.10.3, pages 61-62]{BerghLofstrbook76InterpolatSps})}
For a measure space $(\Omega,\mu)$ with finite measure $\mu$, the $L_p(\mu)$-space of the equivalence classes of measurable functions $x:\Omega\rightarrow \mathbb{R}$, satisfying $\int\limits_{\Omega} \vert x(t)\vert^{p}dt<\infty$ (with the equivalence relation given by equality of functions up to $\mu$-measure zero), is a $b$-metric space with $b=2^{\frac{1}{p}-1}$ \text{ for } $0<p<1$ with the $b$-metric
$$d(x,y)=\big(\int\limits_{\Omega} \vert x(t)-y(t)\vert^{p}dt\big)^{\frac{1}{p}}.
$$
The condition \ref{b3bmetrspdef} follows by replacement of $x$ by $x-z$ and $y$ by $z-y$ for $x,y,z\in L_p(\mu)$ in the following inequality:
$$ \big (\int\limits_{\Omega} \vert x \vert^p d\mu + \int\limits_{\Omega} \vert y \vert^p d\mu\big)^{\frac{1}{p}}
\leq 2^{\frac{1-p}{p}}\big(\big(\int\limits_{\Omega} \vert x \vert^p d\mu\big)^{\frac{1}{p}} + \big(\int\limits_{\Omega} \vert y \vert^p d\mu\big)^{\frac{1}{p}}\big).
$$
\end{example}

\begin{definition}
	\label{NaSeDefinition1.3.3.} 
Let $(X,d)$ be a $b$-metric space. Then a subset $C\subset X$ is called:
	\begin{enumerate}[label=\textup{(\roman*)}, ref=(\roman*)]
		\item closed if and only if for each sequence $\{x_{n}\}$ in $C$ which
		converges to an element $x$, we have $x$ $\in C$ (that is, $C=\overline{C}$).
		
		\item compact if and only if for every sequence of elements of $C$
		there exists a subsequence that converges to an element of $C$.
		
		\item bounded if and only if $\delta (C):=\sup \{d(x,y):x,y\in
		C\}<\infty $.
	\end{enumerate}
\end{definition}

\begin{definition}
	\label{NaSeDefinition1.3.4.} 
Let $(X,d)$ be a $b$-metric space.
	A sequence $\{x_{n}\}$ in $X$ is called:
\begin{enumerate}[label=\textup{(\Roman*)}, ref=(\Roman*)]
	\item Cauchy if and only if for $\varepsilon >0,$ there exists $%
		n(\varepsilon )\in
		\mathbb{Z}_{> 0}$ such that for each $n,m\geq n(\varepsilon )$, we have $d(x_{n},x_{m})<%
		\varepsilon .$		
	\item Convergent if and only if there exists $x\in X$ such that for all $\varepsilon >0$ there exists $n(\varepsilon )\in\mathbb{Z}_{> 0}$ such that for all $n\geq n(\varepsilon )$, we have $d(x_{n},x)<\varepsilon $. In this case, we write $\lim \limits_{n\rightarrow \infty }x_{n}=x.$
	\end{enumerate}
\end{definition}

It is known that a sequence $\{x_{n}\}$ in $b$-metric space $X$ is
Cauchy if and only if $\lim \limits_{n\rightarrow \infty }d(x_{n},x_{n+p})=0$
for all $p\in\mathbb{Z}_{> 0}$. A sequence $\{x_{n}\}$ is convergent to $x\in X$ if and only if $\lim
\limits_{n\rightarrow \infty }d(x_{n},x)=0.$ A $b$-metric space $(X,d)$ is
said to be complete if every Cauchy sequence in $X$ is convergent in $X$.

The $b$-metric spaces have the following topological properties \cite{ATD15}:
\begin{enumerate}[label=\textup{(\alph*)}, ref=(\alph*)]
	\item In a $b$-metric space $(X,d),$ $d$ is not necessarily
	continuous in each variable.
	\item In a $b$-metric space $(X,d),$ an open ball is not
	necessarily an open set. An open ball is open if $d$ is continuous in one variable.
\end{enumerate}

\begin{definition}[\cite{C98}]
	\label{NaSeDefinition1.3.5.}  Let $(X,d)$ be a $b$-metric space and $CB\left( X\right) $ denotes the set of all non-empty closed and bounded
	subsets of $X.$ For $x,y\in X$ and $A\in CB(X),$ the following statements
	hold:
	\begin{enumerate} [label=\textup{\arabic*)}, ref=\arabic*]
		\item $d(x,A)\leq b\left( d(x,y)+d(y,A)\right),$ where $d(x,A)=inf\{d(x,a):a \in A \}$;	
		\item For every $\lambda >0$ and $\tilde{a}\in A,$ there is a $\tilde{b}\in
		B $ such that $d(\tilde{a},\tilde{b})\leq \lambda;$		
		\item $d(x,A)=0$ if and only if $x\in \bar{A}=A.$
	\end{enumerate}	
\end{definition}

\begin{lemma} \label{lemma1iterdistupperest}
	Let $(X,d)$ be a $b$-metric space. For any sequence $\{u_{n}\}$ in $X$,
	\begin{equation*}
		d(u_{0},u_{n})\leq bd(u_{0},u_{1})+\dots+b^{n-1}d(u_{n-2},u_{n-1})+b^{n-1}d(u_{n-1},u_{n}).
	\end{equation*}	
\end{lemma}

  Boriceanu et al. \cite{BBP10} defined multivalued fractals in $b$-metric spaces. Czerwik \cite{CDS97} defined round-off stability of iteration procedures for operators in $b$-metric spaces. Czerwik \cite{C98} obtained nonlinear set-valued contraction mappings in $b$-metric spaces.

\section{Fixed Points of Pre\v{s}i\'c type Contractions in $b$-Metric Spaces}

In 1997, Alber and Guerre-Delabriere \cite{Alber} proved that weakly
contractive mapping defined on a Hilbert space is a Picard operator. Rhoades
\cite{Rhoades} proved that the corresponding result is also valid when
Hilbert space is replaced by a complete metric space. Dutta et al. \cite%
{Dutta} generalized the weak contractive condition and proved a fixed point
theorem for a selfmap. Berinde and P\v{a}curar \cite{Berinde1, Berinde2} established an iterative method for approximating fixed points of Pre\v{s}i\'c contractive mappings.
Chen \cite{CH} defined some Pre\v{s}i\'c type contractive condition and its applications. Khan et al. \cite{K} obtained some convergence results for iterative sequences of Pre\v{s}i\'c type and applications. P\v{a}curar \cite{PA, PA2} approximated common fixed points of Pre\v{s}i\'c-Kannan type operators by a multi-step iterative method. Shukla \cite{Shulka1} obtained some Pre\v{s}i\'c type results in 2-Banach spaces. Shukla and Sen \cite{Shulka2} definded set-valued Pre\v{s}i\'c-Reich type mappings in metric spaces. Recently, Abbas et al. \cite{ATD15} established the fixed point of Pre\v{s}i\'c type mapping satisfying weakly contraction conditions.

In this section, we study several fixed point results for single-valued
Pre\v{s}i\'c type mappings that satisfying generalized contractions in the
framework of $b$-metric spaces. Also an example is given that support the
results proved therein. Our results extend and generalize many comparable
results in the existing literature proved.
We attain fixed point results for single-valued mappings in the framework of $b$-metric spaces along with an example and corollaries. We begin with the
following result.

\begin{theorem}
	\label{NaSeTheorem2.2.1} Let $(X,d)$ be a complete $b$-metric
	space, $k\in \mathbb{Z}_{>0}$ and $f:X^{k}\rightarrow X$ be a given mapping. Suppose that there exists $\phi
	:\mathbb{R}_{\geq 0}\rightarrow \mathbb{R}_{\geq 0}$ a lower semi-continuous function
	having $\phi (t)=0$ if and only if $t=0$ such that
	\begin{equation}
		\begin{array}{l}
			d(f(x_{1},x_{2},\ldots ,x_{k}),f(x_{2},x_{3},\ldots ,x_{k+1}))\\
            \quad \leq \max \{d(x_{i},x_{i+1}):1\leq i\leq k\}-\phi (\max \{d(x_{i},x_{i+1}):1\leq i\leq k\}),
		\end{array}
		\label{t.2.1}
	\end{equation}%
	for all $(x_{1},x_{2},\ldots ,x_{k+1})\in X^{k+1}$. Then, for any arbitrary
	points $x_{1},x_{2},\ldots ,x_{k}\in X$, the sequence $\{x_{n}\}$ defined by
	\eqref{t.1.1} converges to $u\in X$ and $u$ is a fixed point of $f$, that is, $u=f(u,u,\dots,u).$ Moreover, if
	\begin{equation}
		d(f(x,\dots,x),f(y,\dots,y))=d(x,y)-\phi (d(x,y)),  \label{t.2.2}
	\end{equation}%
	holds for all $x,y\in X$ with $x\neq y$, then $u$ is the unique fixed point
	of $f$.
\end{theorem}
\begin{proof}
	Let $x_{1},\ldots ,x_{k}$ be arbitrary $k$ elements in $X$. Let us define
	the sequence $\{x_{n}\}$ in $X$ by
	\begin{equation*}
		x_{n+k}=f(x_{n},\dots ,x_{n+k-1}),\, \,n=1,2,\ldots _{.}
	\end{equation*}%
	If $x_{i}=x_{i+1}$ for $i=n,\dots ,n+k-1$, then $u\in f(u,u,\dots ,u),$ that
	is, $u$ is fixed point of $f.$ So, we suppose that $x_{i}\neq x_{i+1}$ any $%
	i=n,\dots ,n+k-1.$ Now for $n\in \mathbb{Z}_{>0},$ by using \eqref{t.2.1},
	we get the following inequalities:%
	\begin{align*}
		& d(x_{k+1},x_{k+2}) =d(f(x_{1},\ldots ,x_{k}),f(x_{2},\ldots ,x_{k+1})) \\
		&\leq \max \left \{ d(x_{i},x_{i+1}):1\leq i\leq k\right \} -\phi \left( \max
		\left \{ d(x_{i},x_{i+1}):1\leq i\leq k\right \} \right)  \\
		&<\max \left \{ d(x_{i},x_{i+1}):1\leq i\leq k\right \} =\max \{d\left(
		x_{1},x_{2}\right) ,d\left( x_{2},x_{3}\right) ,\ldots,d\left(
		x_{k},x_{k+1}\right) \}. \\
		& d(x_{k+2},x_{k+3}) =d(f(x_{2},\ldots ,x_{k+1}),f(x_{3},\ldots ,x_{k+2})) \\
		&\leq \max \left \{ d(x_{i},x_{i+1}):2\leq i\leq k+1\right \} -\phi \left(
		\max \left \{ d(x_{i},x_{i+1}):2\leq i\leq k+1\right \} \right)  \\
		&<\max \left \{ d(x_{i},x_{i+1}):2\leq i\leq k+1\right \} =\max \{d\left(
		x_{2},x_{3}\right) ,d\left( x_{3},x_{4}\right) ,\ldots,d\left(
		x_{k+1},x_{k+2}\right) \} \\
		&<\max \{d\left( x_{1},x_{2}\right) ,d\left( x_{2},x_{3}\right)
		,\ldots,d\left( x_{k},x_{k+1}\right) \}.
	\end{align*}
	and, by mathematical induction, in general for all integers $n\geq 1$,
	\begin{align*}
		& d(x_{k+n},x_{k+n+1})=d(f(x_{n},\ldots ,x_{k+n-1}),f(x_{n+2},\ldots
		,x_{k+n})) \\
		& \leq \max \left \{ d(x_{i},x_{i+1}):n\leq i\leq n+k-1\right \} -\phi \left(
		\max \left \{ d(x_{i},x_{i+1}):n\leq i\leq n+k-1\right \} \right)  \\
		& <\max \left \{ d(x_{i},x_{i+1}):n\leq i\leq n+k-1\right \} \\
        &=\max \{d\left(
		x_{n},x_{n+1}\right) ,d\left( x_{n+1},x_{n+2}\right),\ldots,
        d\left(x_{n+k-1},x_{n+k}\right) \} \\
		& <\max \{d\left( x_{1},x_{2}\right) ,d\left( x_{2},x_{3}\right)
		,\ldots,d\left( x_{k},x_{k+1}\right) \}.
	\end{align*}
	Thus by combining above, we get
	\begin{eqnarray*}
		\max \left \{ d(x_{i},x_{i+1}):n+1\leq i\leq n+k\right \}  &<&\max \left \{
		d(x_{i},x_{i+1}):n\leq i\leq n+k-1\right \}  \\
		&<&\dots <\max \left \{ d(x_{i},x_{i+1}):1\leq i\leq k\right \}.
	\end{eqnarray*}
We surmise that $\max \left \{ d(x_{i},x_{i+1}):n+1\leq i\leq n+k\right \} $
is monotone decreasing and bounded from below. Hence, there exists some $c\geq 0$ such that
	\begin{equation*}
		\lim\limits_{n\rightarrow \infty }\max \left \{ d(x_{i},x_{i+1}):n+1\leq i\leq
		n+k\right \} =c.
	\end{equation*}
For all $n\in\mathbb{Z}_{\geq 1}$, let $n+1\leq j_n\leq n+k$ be such that
$$ d(x_{j_n},x_{j_n+1})= \max \left \{ d(x_{i},x_{i+1}):n+1\leq i\leq
		n+k\right \}.$$ Then,  	
$\lim\limits_{n\rightarrow \infty }d(x_{j_n},x_{j_{n}+1})=c.$
	
We will show now that $c=0$. Actually, taking the upper limits as
$n\rightarrow \infty $ on both sides of the inequality
\begin{multline*}
d(x_{j_n+k},x_{j_n+k+1})=d(f(x_{j_n},\ldots ,x_{j_n+k-1}),f(x_{j_n+1},\ldots,x_{j_n+k})) \\
\leq \max \left \{ d(x_{i},x_{i+1}):j_n\leq i\leq j_n+k-1\right \} \\
         -\phi \left(\max \{d(x_{i},x_{i+1}):j_n\leq i\leq j_n+k-1\} \right),
\end{multline*}
	yields
	$
		c\leq c-\phi \left( c\right) ,
	$
	that is, $\phi \left( c\right) \leq 0.$ As a result $\phi \left( c\right) =0$
	by the definition of $\phi $, we get $\lim\limits_{n\rightarrow \infty }d(x_{j_n},x_{j_{n}+1})=0$
	and in addition, we get
	\begin{equation}
		\lim\limits_{n\rightarrow \infty }d(x_{n+k},x_{n+k+1})=0.  \label{t.2.3}
	\end{equation}%
	Next we verify that $\{x_{n+k}\}$ is Cauchy. For any $n,p\in \mathbb{Z}_{>0}$
	with $n>p$ and by using \eqref{t.2.1}, Lemma \ref{lemma1iterdistupperest} and $b\geq 1$, we get
	\begin{align*}
		& d(x_{n+k},x_{n+k+p}) =d(f(x_{n},\ldots ,x_{n+k-1}),f(x_{n+p},\ldots ,x_{n+p+k-1})) \\
		& \leq \sum\limits_{j=1}^{p} b^{j} d(f(x_{n+j-1},\ldots ,x_{n+j+k-2}),f(x_{n+j},\ldots
		,x_{n+j+k-1}))
		\\
		& \leq \sum\limits_{j=1}^{p} b^{j} (\max \left \{ d(x_{i},x_{i+1}):n+j\leq i\leq n+j+k-1\right \}
		\\
		& \hspace{2cm}-\phi \left( \max \{d(x_{i},x_{i+1}):n+j\leq i\leq
		n+j+k-1\} \right) ).
	\end{align*}
	From this we conclude that $\{x_{n+k}\}$ is a Cauchy sequence in $(X,d)$.
	Since $(X,d)$ is complete $b$-metric, there exists a $u$ in $X$ such that
	\begin{equation}
		\lim\limits_{n\rightarrow \infty }d(x_{n+k},x_{n+k+p})=\lim\limits_{n\rightarrow \infty
		}d(x_{n+k},u).  \label{t.2.4}
	\end{equation}%
	For any $n\in \mathbb{R}_{>0}$, we have
	\begin{align*}
		& d(u,f(u,u,\ldots ,u)) \leq bd(u,x_{n+k})+bd(x_{n+k},f(u,u,\ldots ,u)) \\
		&= bd(u,x_{n+k})+bd(f(u,u,\ldots ,u),f(x_{n},x_{n+1},\ldots ,x_{n+k-1}))
		\\
		& \leq bd(u,x_{n+k})+\sum\limits_{j=1}^{k} b^{j+1}d(f(\underbrace{u,\dots,u}_{k-j+1},\underbrace{x_{n},\dots,x_{n+j-2}}_{j-1}),f(\underbrace{u,\dots,u}_{k-j},\underbrace{x_{n},\dots,x_{n+j-1}}_{j}))\\
		& \leq
bd(u,x_{n+k})+ \sum\limits_{j=1}^{k} b^{j+1}(\max \{d(u,x_{n}),d(x_{n},x_{n+1}),\ldots ,d(x_{n+j-2},x_{n+j-1})\} \\
        & \hspace{3cm}  - \phi(\max \{d(u,x_{n}),d(x_{n},x_{n+1}),\ldots
		,d(x_{n+j-2},x_{n+j-1})\})).
	\end{align*}%
	Now, by taking upper limit as $n\rightarrow \infty $ in the upper inequality
	and by using \eqref{t.2.4}, we get
	\[
	d(u,f(u,u,\ldots ,u))\leq 0,
	\]%
	which suggest that $u=f(u,u,\ldots ,u)$. So $u$ is a fixed point of $f$.
	
	To prove the fixed point uniqueness, suppose that there exists $v\in X$ with
	$v\neq u$, such that $v=f(v,v,\dots,v)$. Then by \eqref{t.2.2}, we get
	\begin{eqnarray*}
		d(u,v) &=&d(f(u,u,\dots,u),f(v,v,\dots,v)) \leq d(u,v)-\phi (d(u,v)) <d(u,v),
	\end{eqnarray*}
	a contradiction. Thus, $u$ is the unique point in $X$ and
$u=f(u,u,\dots,u)$.
\qed \end{proof}

\begin{example}
	\label{NaSeExample2.1.2.} Let $X=[0,2]$ and $d(x,y)=(x-y)^{2}$ be a $b$-metric on $X$. Let $k\geq 1$ be a positive integer and $%
	f:X^{k}\rightarrow X$ be the mapping defined by
	\begin{equation*}
		f(x_{1},\dots,x_{k})=\frac{x_{1}+\dots+x_{k}}{2k}\text{ for all }
		x_{1},\dots,x_{k}\in X.
	\end{equation*}    	
Define $\phi :\mathbb{R}_{\geq 0} \rightarrow \mathbb{R}_{\geq 0}$ by
	\begin{equation*}
		\phi (t)=\left \{
		\begin{array}{l}
			\text{ \  \ }\frac{t}{5},\text{ \  \  \  \ if }t\in \lbrack 0,\frac{5}{2}), \\
			\\
			\frac{2^{2n}(2^{n+1}t-3)}{2^{2n+1}-1},\text{ \  \ if }t\in \lbrack
            \frac{2^{2n}+1}{2^{n}},\frac{2^{2(n+1)}+1}{2^{n+1}}],\text{ }n\in
			\mathbb{Z}_{> 0}.
		\end{array}%
		\right.
	\end{equation*}
	An easy computation shows that $\phi $ is lower semi-continuous on $\mathbb{R}_{\geq 0}$ and $\phi (t)=0$ if and only if $t=0.$
	
	  Now, for all $x_{1},\dots,x_{k+1}\in X$, we have
	\begin{eqnarray*}
		&d(f(x_{1},\dots,x_{k}),f(x_{2},\dots,x_{k+1})) \leq \dfrac{\vert x_{1}-x_{k+1}\vert^{2}}{4k^{2}} \leq
        \frac{1}{4}\max\limits_{1\leq i\leq k}\{ \left \vert x_{i}-x_{i+1}\right
		\vert ^{2}\} \\
		&\leq \frac{4}{5}\max\limits_{1\leq i\leq k}\{d(x_{i},x_{i+1})\} =\max\limits_{1\leq i\leq k}\{d(x_{i},x_{i+1})\}-\phi (\max\limits_{1\leq i\leq k}\{d(x_{i},x_{i+1})\}).
	\end{eqnarray*}
	Also, for all $x_{1},x_{2}\in X$, we have
	\begin{eqnarray*}
		d(f(x_{1},\dots,x_{1}),f(x_{2},\dots,x_{2})) &\leq &\frac{1}{8}\left
		\vert x_{1}-x_{2}\right \vert  \\
        &\leq & \frac{4}{5}d(x_{1},x_{2}) =d(x_{1},x_{2})-\phi (d(x_{1},x_{2})).
	\end{eqnarray*}
	Hence $f$ satisfies \eqref{t.2.1} and \eqref{t.2.2}. Thus, all the required hypotheses of
	Theorem \ref{NaSeTheorem2.2.1} are satisfied. Moreover, for any arbitrary points
$x_{1},\dots,x_{k}\in X$, the sequence $\{x_{n}\}$ defined by
$x_{n+k}=f(x_{n},\dots,x_{n+k-1}),\ n=1,2,\ldots \ $converges to $u=0$,
	the unique fixed point of $f$.
\qed \end{example}

\begin{corollary}
	\label{NaSeCorollary2.1.3.} Let $(X,d)$ be a complete $b$-metric space, $k\in\mathbb{Z}_{>0}$ and $f:X^{k}\rightarrow X$ be a given mapping.
If there exists $%
	\lambda \in \lbrack 0,1)$ such that
	\begin{equation}
		d(f(x_{1},\ldots ,x_{k}),f(x_{2},\ldots ,x_{k+1}))\leq \lambda \max
		\{d(x_{i},x_{i+1}):1\leq i\leq k\},  \label{t.2.5}
	\end{equation}
	for all $(x_{1},\ldots ,x_{k+1})\in X^{k+1}$, then, for any $x_{1},\ldots,x_{k}\in X$, the sequence $\{x_{n}\}$ defined in \eqref{t.1.1} converges to $u$, and
	also $u$ is a fixed point of $f$ and $u=f(u,\ldots ,u)$. Moreover, if
	\begin{equation*}
		d(f(x,\ldots ,x),f(y,\ldots ,y))\leq \lambda d(x,y),
	\end{equation*}
	holds for all $x,y\in X$ and $x\neq y$, then $u$ is the unique fixed point
	of $f$.
\end{corollary}

\begin{theorem}
	\label{NaSeTheorem2.1.4} Let $(X,d)$ be a complete $b$-metric space, $k\in
	\mathbb{Z}_{>0}$ and $f:X^{k}\rightarrow X$ be a given mapping. Let us consider that
	there exist $a\in \mathbb{R}$ which is a constant such that $0\leq akb^{k+1}<1$ and
	\begin{equation}
		d(f(x_{1},\dots,x_{k}),f(x_{2},\dots,x_{k+1}))\leq a\max_{1\leq i\leq
			k+1}\{d(x_{i},f(x_{i},\dots,x_{i}))\}  \label{t.2.6}
	\end{equation}%
	for all $(x_{1},\dots,x_{k+1})\in X^{k+1}$. Then,
	\begin{enumerate}[label=\textup{(\roman*)}, ref=(\roman*)]
		\item $f$ contain a unique fixed point $u\in X$,
	
	\item For any points $x_{1},\dots,x_{k}\in X,$ the sequence $\{x_{n}\}$ defined in \eqref{t.1.1} converges to $u$.
\end{enumerate}
\end{theorem}

\begin{proof} \smartqed Let us define a mapping $F:X\rightarrow X$ by $%
	F(x)=f(x,x,\dots,x),$ for all $x\in X$. \\
    For $x,y\in X$, we have
	\begin{eqnarray*}
		d(F(x),F(y)) &\leq &d(f(x,\dots,x),f(y,\dots,y)) \\
		&\leq & \sum\limits_{j=1}^{k} b^{j}d(f(\underbrace{x,\dots,x}_{k-j+1},\underbrace{y,\dots,y}_{j-1})),f(\underbrace{x,\dots,x}_{k-j},\underbrace{y,\dots,y}_{j})).
	\end{eqnarray*}
	Then from \eqref{t.2.6},
	\begin{eqnarray*}
		d(F(x),F(y)) &\leq & a \sum\limits_{j=1}^{k} b^{j} \max\{d(x,f(x,\dots,x),d(y,f(y,\dots,y)\} \\
		&\leq &a(b^{k}+\dots+b^{k})\max \{d(x,f(x,\dots,x)),d(y,f(y,\dots,y))\} \\
		&\leq &akb^{k}\max \{d(x,f(x,\dots,x)),d(y,f(y,\dots,y))\},
	\end{eqnarray*}
	and we get
	\begin{equation}
		d(F(x),F(y))\leq \lambda \max \{d(x,F(x)),d(y,F(y))\},  \label{t.2.7}
	\end{equation}
	where $\lambda =akb^{k}\in \lbrack 0,1)$. Now, for any $x_{0}\in X,$ we
	define $x_{n+1}=F\left( x_{n}\right) $ for $n \in \mathbb{Z}_{\geq 0}$. If $x_{l}=x_{l+1}$ for some $l\in \mathbb{Z}_{\geq 0},$ then we get $x_{l}=F\left(
	x_{l}\right) ,$ that is, $x_{l}=f(x_{l},\dots,x_{l}).$ We assume that $x_{l+1}\neq x_{l}$ for all $l\in \mathbb{Z}_{\geq 0}$. From \eqref{t.2.7},
	\begin{multline*}
		d\left( x_{n},x_{n+1}\right) =d(F(x_{n-1}),F(x_{n})) \\
		\leq \lambda \max \{d(x_{n-1},F(x_{n-1})),d(x_{n},F(x_{n}))\} =\lambda \max \{d(x_{n-1},x_{n}),d(x_{n},x_{n+1})\}.
	\end{multline*}
	If $\max \{d(x_{n-1},x_{n}),d(x_{n},x_{n+1})\}=d(x_{n},x_{n+1}),$ then we
	obtain
	\begin{equation*}
		d\left( x_{n},x_{n+1}\right) \leq \lambda d(x_{n},x_{n+1}),
	\end{equation*}
	which gives $x_{n}=x_{n+1}$ as $\lambda <1,$ a contradiction. Hence,
$$\max\{d(x_{n-1},x_{n}),d(x_{n},x_{n+1})\}=d(x_{n-1},x_{n}),$$
and for all $n \in \mathbb{Z}_{> 0}, $ we obtain
	\begin{equation*}
		d\left( x_{n},x_{n+1}\right) \leq \dots\leq \lambda^j d(x_{n-j},x_{j+1}) \leq \dots \leq \lambda ^{n}d\left( x_{0},x_{1}\right) .
	\end{equation*}
	Now for $m,n\in \mathbb{Z}_{> 0}	$ with $m> n,$ we have
\begin{eqnarray*}
		d\left( x_{n},x_{m}\right) &\leq &  \sum\limits_{j=1}^{m-n} b^{j}d\left(x_{n+j-1},x_{n+j}\right)  \leq	\sum\limits_{j=1}^{m-n} b^{j}\lambda^{n+j-1}d(x_{0},x_{1}) \\
		&= &b\lambda ^{n}\left(\sum\limits_{j=0}^{m-n-1}\left( b\lambda \right)^{j}\right) d\left( x_{0},x_{1}\right) = b\lambda ^{n}\frac{1-\left( b\lambda \right)^{m-n}}{1-b\lambda} d\left( x_{0},x_{1}\right) \\
		&\leq &\frac{(b\lambda) ^{n}}{1-b\lambda}d\left( x_{0},x_{1}\right)
	\end{eqnarray*}
	and on taking limit as $n,m\rightarrow \infty $ implies $d\left(
	x_{n},x_{m}\right) \rightarrow 0.$ Hence the sequence $\{x_{n}\}$ is Cauchy
	and their exists a $u\in X,$ such that $\{x_{n}\}$ converges to $u.$ Now again
	from \eqref{t.2.7}, we have
	\begin{eqnarray*}
		d\left( x_{n},F\left( u\right) \right) &=&d\left( F\left( x_{n-1}\right)
		,F\left( u\right) \right) \\
		&\leq &\lambda \max \{d(x_{n-1},F(x_{n-1})),d(u,F(u))\}
	\end{eqnarray*}
	and taking limit as $n\rightarrow \infty $ yields that
	$d\left( u,F\left( u\right) \right) \leq \lambda d\left( u,F\left( u\right)\right),$
	which gives $u=F\left( u\right),$ that is, $u=f(u,\dots,u).$
	Now, assume that their exists another $v\in X$ such that $v=f(v,\dots,v).$
	From \eqref{t.2.7}, we obtain
	\begin{eqnarray*}
		d\left( u,v\right) &=&d\left( F\left( u\right) ,F\left( v\right) \right) \\
		&\leq &\lambda \max \{d\left( u,F\left( u\right) \right) ,d\left( v,F\left(
		v\right) \right) \} =\lambda \max \{d\left( u,v\right) ,d\left( v,v\right) \}=0,
	\end{eqnarray*}
	which implies that uniqueness of fixed point $u=Fu=f(u,\dots,u)$ of $f.$

	For arbitrary points $x_{1},\dots,x_{k}\in X,$ we are going to prove the
	convergence of the sequence $\{x_{n}\}$ defined by \eqref{t.1.1} to $u$, the unique	fixed point of $f$. For all $n\geq k+1$, we get
	\begin{equation*}
		x_{n}=f(x_{n-k},\dots.,x_{n-1}).
	\end{equation*}
	As we know that $f$ has unique fixed point $u\in X,$ we can write
	\begin{eqnarray*}
		d(x_{n+1},u) &=&d(f(x_{n-k+1},x_{n-k+2},\dots,x_{n}),f(u,\dots,u)) \\
		&\leq &
		\sum\limits_{j=1}^{k} b^{j}d(f(x_{n-k+j},\dots,x_{n},\underbrace{u,\dots,u}_{j-1}),f(x_{n-k+j+1},\dots,x_{n},\underbrace{u,\dots,u}_{j})).\\
	\end{eqnarray*}
	This suggest from \eqref{t.2.6} that
	\begin{eqnarray*}
		d(x_{n+1},u) &\leq &
		a\sum\limits_{j=1}^{k} b^{j}\max\{{d(x_{n-k+j},F(x_{n-k+j})),\dots,d(x_{n},F(x_{n}))},\\ & & \hspace{4cm} \underbrace{d(u,F(u)),\dots,d(u,F(u))}_{j}\} \\
		&\leq& 	ab^k\sum\limits_{j=1}^{k}\max\{{d(x_{n-k+j},F(x_{n-k+j})),\dots,d(x_{n},F(x_{n}))},\\ & & \hspace{4cm}
\underbrace{d(u,F(u)),\dots,d(u,F(u))}_{j}\}.
	\end{eqnarray*}
	Since $u=F(u)$, we get
	\begin{eqnarray}
		d(x_{n+1},u) &\leq & ab^k\sum\limits_{j=1}^{k}\max\{{d(x_{n-k+j},F(x_{n-k+j})),\dots,d(x_{n},F(x_{n}))}\}.  \label{t.2.8}
	\end{eqnarray}
	For all $j\in \mathbb{Z}_{> 0}	$, we have
	\begin{equation}
		d(x_{j},F(x_{j}))\leq b[d(x_{j},u)+d(u,F(x_{j}))].  \label{t.2.9}
	\end{equation}
	By \eqref{t.2.7}, we possess
	\begin{eqnarray*}
		d(u,F(x_{j})) &=&d(F(u),F(x_{j})) \\
		&\leq &\lambda \max \{d(u,F(u)),d(x_{j},F(x_{j})\} \\
		&=&\lambda d(x_{j},F(x_{j})).
	\end{eqnarray*}
	Thus \eqref{t.2.9} becomes
	\begin{equation*}
		d(x_{j},F(x_{j}))\leq b[d(x_{j},u)+\lambda d(x_{j},F(x_{j}))].
	\end{equation*}
	which provides%
	\begin{equation}
		d(x_{j},F(x_{j}))\leq \frac{b}{1-\lambda b}d(x_{j},u)\text{, for all }j\in \mathbb{Z}_{>0}.  \label{t.2.10}
	\end{equation}%
	Using \eqref{t.2.8} and \eqref{t.2.10}, we obtain 	
\begin{eqnarray}
		d(x_{n+1},u) &\leq &\frac{ab}{1-\lambda b}\max
		\{d(x_{n-k+1},u),\dots,d(x_{n},u)\}  \notag \\
		&&+\dots+\frac{ab}{1-\lambda b}\max \{d(x_{n-k+i},u),\dots,d(x_{n},u)\}  \notag \\
		&&+\dots+\frac{ab}{1-\lambda b}d(x_{n},u)  \notag \\
		&\leq &\frac{abk}{1-\lambda b}\max \{d(x_{n-k+1},u),\dots,d(x_{n},u)\}  \notag
	\end{eqnarray}%
	for all $n\geq k$. Setting
	\begin{align*}
		\zeta _{n}=d(x_{n},u)\text{, for all }n\in \mathbb{Z}_{> 0},\\
		\alpha =\frac{abk}{1-\lambda b},
	\end{align*}
	we obtain
	$
		\zeta _{n+1}\leq \alpha \max \{ \xi _{n-k+1},\dots,\zeta _{n}\},
	$
	for all $n\geq k$. We have $0\leq \alpha <1$. Following the same arguments
	from \cite[Lemma 2]{Presic}, there exists $L>0$ and $\theta \in (0,1)$ such
	that $\zeta _{n}\leq L\theta ^{n}$ for all $n\in
	\mathbb{Z}_{> 0},$ namely such that
	\begin{equation*}
		d(x_{n},u)\leq L\theta ^{n}\text{ for all }n\geq 1.
	\end{equation*}
	Taking limit as $n\rightarrow \infty $ in the above inequality, we get
$\lim\limits_{n\rightarrow \infty }d(x_{n},u)=0$, So sequence $\{x_{n}\}$ converges to a unique fixed point of $f$.
	\qed \end{proof}

\begin{theorem}
	\label{NaSeTheorem2.1.5.} Let $(X,d)$ be a complete $b$-metric
	space, $k\in\mathbb{Z}_{>0}$ and $f:X^{k}\rightarrow X$ a mapping fulfilling the following
	contractive type condition%
	\begin{equation*}
		d(f(x_{1},\dots,x_{k}),f(x_{2},\dots,x_{k+1}))\leq \eta \max_{1\leq i\leq
			k}\{d(x_{i},x_{i+1})\},
	\end{equation*}
	where $\eta \in (0,1)$ a constant and $x_{1},\dots,x_{k+1}$ are the
	arbitrary elements in $X$. Then there exists $x\in X$ such that $f(x,\dots,x)=x$. Furthermore, if $x_{1},\dots,x_{k}$ are arbitrary points in $X$ and for $n\in \mathbb{Z}_{> 0},$
	\begin{equation*}
		x_{n+k}=f(x_{n},\dots,x_{n+k-1}),
	\end{equation*}
	then the sequence taken $\{x_{n}\}_{n=1}^{\infty }$ is convergent and%
	\begin{equation*}
		\lim\limits_{n\rightarrow \infty} x_{n}=f(\lim\limits_{n\rightarrow \infty} x_{n},\dots,\lim\limits_{n\rightarrow \infty} x_{n}).
	\end{equation*}
	Additionally we suppose that, on the diagonal $\triangle \subset X^{k},$
	\begin{equation*}
		d(f(u,\dots,u),f(v,\dots,v))<d(u,v)
	\end{equation*}
	holds for every $u,v\in X,$ and $u\neq v$. Then $x$ is the unique point in $X $ with $f(x,\dots,x)=x.$
\end{theorem}

\begin{proof} \smartqed Let $x_{1},\dots,x_{k}$ be $k$ arbitrary points in $X$. By using these points, we define a sequence $\{x_{n}\}$ by
	\begin{equation*}
		x_{n+k}=f(x_{n},\dots,x_{n+k-1})
	\end{equation*}
	for $n=1,2,\dots$. For simplicity of notation set $\alpha _{n}=d(x_{n},x_{n+1}).$ We are going to prove by
	induction that for each $n\in \mathbb{Z}_{> 0},$
	\begin{equation}
		\alpha _{n}\leq b^{k}K\theta ^{n}\ \text{where}\ \theta =\eta^{\frac{1}{k}},\
K=\max \{ \frac{\alpha_{1}}{\theta },\dots,\frac{\alpha_{k}}{\theta^{k}}\}.  \label{t.2.15}
	\end{equation}
	From the definition of $K$, we see that \eqref{t.2.15} is true for $n=1,\dots,k.$ Now
	suppose the following $k$ inequalities:
	\begin{equation*}
		\alpha _{n}\leq b^{k}K\theta ^{n},\text{ }\alpha _{n+1}\leq b^{k}K\theta
		^{n+1},\dots,\text{ }\alpha _{n+k-1}\leq b^{k}K\theta ^{n+k-1}. \label{t.2.16}
	\end{equation*}
	be the induction hypothesis. Then,
	\begin{eqnarray*}
		\alpha _{n+k} &=&d(x_{n+k},x_{n+k-1}) \\
		&=&d(f(x_{n},x_{n+1},\dots,x_{n+k-1}),f(x_{n+1},x_{n+2},\dots,x_{n+k})) \\
		&\leq &\eta \max \{ \alpha _{n},\alpha _{n+1},\dots,\alpha _{n+k-1}\} \quad  \text{(by \eqref{t.2.16} and definition of } \alpha _{i}) \\
		&\leq &\eta \max \{b^{k}K\theta ^{n},b^{k}K\theta ^{n+1},\dots,b^{k}K\theta
		^{n+k-1}\}\ \ \text{(by induction hypothesis)} \\
		&=&\eta b^{k}K\theta^{n}\qquad \text{(as\ $0\leq \theta <1$)} \\
		&=&b^{k}K\theta ^{n+k} \qquad \text{(as\ $\theta =\eta^{\frac{1}{k}}$)}
	\end{eqnarray*}
	and it completes the inductive proof of \eqref{t.2.15}. Next by using \eqref{t.2.15} for
	any $n,p\in \mathbb{Z}_{> 0},$ we have
	\begin{align*}
		& d(x_{n},x_{n+p}) \leq  d(x_{n},x_{n+1})+\dots+b^{j}d(x_{n+j-1},x_{n+j})+\dots+b^{p}d(x_{n+p-1},x_{n+p}) \\
		&\leq
		b^{p}d(x_{n},x_{n+1})+\dots+b^{p}d(x_{n+j-1},x_{n+j})+\dots+b^{p}d(x_{n+p-1},x_{n+p})
		\\
		&\leq b^{p}K\theta ^{n}+\dots+b^{p}K\theta ^{n+j-1}+\dots+b^{p}K\theta ^{n+p-1} \\
		&\leq b^{p}K\theta ^{n}(1+\theta +\dots+\theta ^{n+p-1}) \\
		&=\frac{b^{p}K\theta ^{n}}{1-\theta }.
	\end{align*}
	From this we conclude that $\{x_{n}\}$ is a Cauchy sequence. Since $X$ is complete space, there exists $x\in X$ such that
	$
		x=\lim\limits_{n\rightarrow \infty }x_{n}.
	$
	Then for any integer $n$, we may write
		\begin{align*}
		&d(x,f(x,\dots,x)) \leq b(d(x,x_{n+k})+d(x_{n+k},f(x,\dots,x))) \\
		&= bd(x,x_{n+k})+d(f(x,\dots,x),f(x_{n},\dots,x_{n+k-1}))) \\
		&\leq  bd(x,x_{n+k})+ \sum\limits_{j=1}^{k}  b^{j}d(f(\underbrace{x,\dots,x}_{k-j+
			1},\underbrace{x_{n},\dots,x_{n+j-2},}_{j-1}),f(\underbrace{x,\dots,x}_{k-j},\underbrace{x_{n},\dots,x_{n+j-1},}_{j}))\\
		& \leq bd(x,x_{n+k})+ \eta \sum\limits_{j=1}^{k} b^{j} \max	\{d(x,x_{n}),d(x_{n},x_{n+1}),\dots,d(x_{n+j-2},x_{n+j-1})\}  .
			\end{align*}
In the limit, when $n\rightarrow\infty$, we get $d(x,f(x,\dots,x))\leq 0$, which gives $f(x,\dots,x)=x.$ Thus, we proved that
	\begin{equation*}
		\lim\limits_{n\rightarrow \infty} x_{n}=f(\lim\limits_{n\rightarrow \infty} x_{n},\dots,\lim\limits_{n\rightarrow \infty} x_{n}). \
	\end{equation*}
	Now consider that $f(x,\dots,x)=x$ holds. To prove that the fixed point is unique,
	Suppose that for some $y\in X$, $y\neq x$, we have $f(y,\dots,y)=y.$ Then by $d(x,y)=d(f(x,\dots,x),f(y,\dots,y))<d(x,y),$ which is a contradiction.
	So, $x$ is the unique point in $X$ such that $f(x,x,\dots,x)=x.$
\qed \end{proof}


\end{document}